\documentclass{birkjour}
 \newtheorem{thm}{Theorem}[section]
 \newtheorem{cor}[thm]{Corollary}
 \newtheorem{lem}[thm]{Lemma}
 
 \theoremstyle{definition}
 
 \theoremstyle{remark}

 \numberwithin{equation}{section}

\begin{document}

%-------------------------------------------------------------------------
% editorial commands: to be inserted by the editorial office
%
%\firstpage{1} \volume{228} \Copyrightyear{2004} \DOI{003-0001}
%
%
%\seriesextra{Just an add-on}
%\seriesextraline{This is the Concrete Title of this Book\br H.E. R and S.T.C. W, Eds.}
%
% for journals:
%
%\firstpage{1}
%\issuenumber{1}
%\Volumeandyear{1 (2004)}
%\Copyrightyear{2004}
%\DOI{003-xxxx-y}
%\Signet
%\commby{inhouse}
%\submitted{March 14, 2003}
%\received{March 16, 2000}
%\revised{June 1, 2000}
%\accepted{July 22, 2000}
%
%
%
%---------------------------------------------------------------------------
%Insert here the title, affiliations and abstract:
%

\title[Complete bounded $\lambda$-hypersurfaces]
 {Complete bounded $\lambda$-hypersurfaces in the weighted volume-preserving mean curvature flow}

%----------Author 1
\author[Yecheng Zhu]{Yecheng Zhu}

\address{%
1. Department of Mathematics, University of Science and Technology of China,
230026, Hefei, Anhui Province, People's Republic of China \\
2. Department of Applied Mathematics, Anhui University of Technology,
243002, Maanshan, Anhui Province, People's Republic of China }
\email{
zhuyc929@mail.ustc.edu.cn }

\thanks{This work was supported by the National Natural Science Foundation of China (No. 11271343). }

%----------Author 2
\author{Yi Fang}
\address{Department of Applied Mathematics, Anhui University of Technology,
243002, Maanshan, Anhui Province, People's Republic of China}
\email{yif1915@ahut.edu.cn}
%----------Author 3
\author{Qing Chen}
\address{Department of Mathematics, University of Science and Technology of China,
230026, Hefei, Anhui Province, People's Republic of China}
\email{qchen@ustc.edu.cn}

%----------classification, keywords, date


\date{September 24, 2016}
%----------additions
%\dedicatory{To my boss}
%%% ----------------------------------------------------------------------

\begin{abstract}
In this paper, we study the complete bounded $\lambda$-hypersurfaces in weighted volume-preserving mean curvature flow. Firstly, we investigate the volume comparison theorem of complete bounded $\lambda$-hypersurfaces with $|A|\leq\alpha$ and get some applications of the volume comparison theorem. Secondly, we consider the relation among $\lambda$, extrinsic radius $k$, intrinsic diameter $d$, and dimension $n$ of the complete $\lambda$-hypersurface, and we obtain some estimates for the intrinsic diameter and the extrinsic radius. At last, we get some topological properties of the bounded $\lambda$-hypersurface with some natural and general restrictions.
\end{abstract}

%%% ----------------------------------------------------------------------
\maketitle
%%% ----------------------------------------------------------------------
%\tableofcontents
\section{Introduction}
A hypersurface $X : M^n \rightarrow R^{n+1}$ is said to be a self-shrinker in $R^{n+1}$ if it satisfies the following equation (see [10]) for the mean curvature and the normal
\begin{equation}
H-\frac{<X, N>}{2}=0.
\end{equation}
Self-shrinkers play an important role in the study of the mean curvature flow. Not only they correspond to the self-shrinking solutions to mean
curvature flow, but also they describe all possible blow ups at a given singularity points of the mean curvature flow. The simplest self-shrinkers are $R^n$, the sphere of radius $\sqrt{2n}$ and more generally cylindrical products $S^k \times R^{n-k}$ (where $S^k$ has radius $\sqrt{2k}$). All of these examples are mean
convex. Without the assumption on mean convexity, there are expected to be many more examples of self-shrinkers in $R^3$. In particular, Angenent constructed a self-shrinking torus (¡°shrinking donut¡±) of revolution in [2], and there is numerical evidence for a number of other examples (see [3], [9], [24]). We refer the readers to [10, 11, 15, 16] and references therein for more information on self-shrinkers and singularities of mean curvature flow.

As generalizations of self-shrinkers, $\lambda$-hypersurfaces were first introduced by Cheng and Wei in [8], where they
proved that $\lambda$-hypersurfaces are critical points of the weighted area functional for the weighted
volume-preserving variations. Furthermore, they classified the complete $\lambda$-hypersurfaces with polynomial volume growth and studied F-stability of $\lambda$-hypersurfaces, which are generalizations of the results due to Huisken [15] and Colding-Minicozzi [10]. Guang proved some gap theorems and Bernstein type theorems for complete $\lambda$-hypersurfaces with polynomial volume growth in terms of the norm of the second fundamental form in [14].
More results on $\lambda$-hypersurfaces can be found in [6, 19, 25, etc.].

We follow the notations of [14, 19] and call a hypersurface $X: M^n \rightarrow R^{n+1}$ a $\lambda$-hypersurface if
it satisfies the curvature condition
\begin{equation}
H-\frac{<N,X>}{2}=\lambda,
\end{equation}
where $\lambda$ is a constant, $N$ is the unit normal vector of $X$ and $H$ is the mean curvature of $M$. One can prove that $\lambda$-hypersurface is a hypersurface with constant mean curvature $\lambda$ in $R^{n+1}$ with respect to the metric $g_{ij} = e^{-\frac{|X|^2}{4}}\delta_{ij}$.

In this paper, we study the volume comparison theorem and topology of complete bounded $\lambda$-hypersurfaces in the weighted volume-preserving mean curvature flow. The organization of this article is as follows:
In section 2, we recall some backgrounds and derive some formulas for $\lambda$-hypersurfaces.
In section 3, we investigate the volume comparison theorem of complete bounded $\lambda$-hypersurfaces with $|A|\leq\alpha$.
In section 4, we give some applications of the volume comparison theorem of $\lambda$-hypersurfaces.
In section 5, we study the relation among $\lambda$, the radius $k$ and the dimension $n$, for the complete $\lambda$-hypersurfaces  with controlled intrinsic
volume growth contained in the Euclidean closed ball $\overline{\mathcal{B}}^{n+1}_k(0)$.
In section 6, we generalize the well-known Myers' theorem on a complete and connected $\lambda$-hypersurface with $Ric_f \geq \frac{1-\lambda^2-3\alpha^2}{2} >0$.
In section 7, we obtain some properties on the topology at the infinity of a bounded $\lambda$-hypersurface with $Ric_f\geq 0$.
In section 8, we get some natural and general restrictions that force the $\lambda$-hypersurface to be compact.

\section{ Preliminaries}
Throughout this paper, the Einstein convention of summing over repeated indices from
1 to $n$ will be adopted.

Let $X:M^n \rightarrow R^{n+1}$ be an $n$-dimensional complete hypersurface of
Euclidean space $R^{n+1}$. We choose a local orthonormal frame field $\{e_\mathcal{\epsilon}\}^{n+1}_{\mathcal{\epsilon}=1}$ in $R^{n+1}$
with dual coframe field $\{\omega_\mathcal{\epsilon}\}^{n+1}_{\mathcal{\epsilon}=1}$ , such that, restricted to $M^n$, $e_1, \cdots , e_n$ are tangent
to $M^n$, and $e_{n+1}$ is the unit normal vector $N$. The coefficients of the second fundamental form $A$ are defined to be
\begin{equation}
h_{ij} = <\nabla_{e_i}e_j, N> .
\end{equation}
In particular, we have
\begin{equation}
\nabla_{e_i}N = -h_{ij}e_j.
\end{equation}
Since $<\nabla_NN, N>=0$, then the mean curvature
\begin{equation}
H = <\nabla_{e_i}N, e_i>=-h_{ii}.
\end{equation}
The Riemann curvature tensor and the Ricci tensor are
given by Gauss equation
\begin{equation}
R_{ijkl} = h_{ik}h_{jl} - h_{il}h_{jk},
\end{equation}
\begin{equation}
R_{ij} = -Hh_{ij} - h_{il}h_{lj}.
\end{equation}
Let $f = \frac{|X|^2}{4}$,
and denote by $d vol_f$ the corresponding weighted volume measure of $M$,
\begin{equation}
dvol_f = e^{-f} dvol.
\end{equation}
Thus, $M = (M, dvol_f)$ is a smooth metric measure space. There is a natural drifted Laplacian on $(M, dvol_f)$ defined by
\begin{equation}
\triangle_f = e^f div(e^{-f}\nabla) =\triangle -<\nabla ,\nabla f>.
\end{equation}
The $\infty$ - Bakry - Emery Ricci tensor $Ric_f$ of $(M, dvol_f)$ is defined by
\begin{equation}
Ric_f = Ric + Hess(f).
\end{equation}
Next we look at the $\infty$ - Bakry - Emery Ricci tensor $Ric_f$ of $\lambda$-hypersurface. For simplicity, we choose a frame such that $\nabla^T_{e_i}e_j=0$, then

\begin{eqnarray}
& &(Ric_f)_{ij}=R_{ij}+\nabla_{e_i}\nabla_{e_j} f\nonumber \\
& &\ \ \ \ \ \ \ \ \ \ \ =R_{ij}+ \frac{1}{2}(<\nabla_{e_i}X^T,e_j>+ <X^T,\nabla_{e_i}e_j>) \nonumber \\
& &\ \ \ \ \ \ \ \ \ \ \ =R_{ij}+ \frac{1}{2}(<\nabla_{e_i}X,e_j>- <\nabla_{e_i}(<X,N>N),e_j> \nonumber \\
& &\ \ \ \ \ \ \ \ \ \ \ \ \ \ + <X,\nabla_{e_i}e_j>-<X,N><N,\nabla_{e_i}e_j> )\nonumber \\
& &\ \ \ \ \ \ \ \ \ \ \ =R_{ij}+ \frac{1}{2}(\delta_{ij}-<X,N><\nabla_{e_i}N,e_j>   \nonumber\\
& &\ \ \ \ \ \ \ \ \ \ \ \ \ \ +<X,<\nabla_{e_i}e_j,N>N>-<X,N><\nabla_{e_i}e_j,N>)                \nonumber\\
& &\ \ \ \ \ \ \ \ \ \ \ =R_{ij}+ \frac{1}{2}\delta_{ij}-\frac{1}{2}<X,N><\nabla_{e_i}N,e_j>   \nonumber\\
& &\ \ \ \ \ \ \ \ \ \ \ =R_{ij}+ \frac{1}{2}\delta_{ij}+\frac{1}{2}<X,N>h_{ij}   \nonumber\\
& &\ \ \ \ \ \ \ \ \ \ \ =-Hh_{ij} - h_{il}h_{lj}+\frac{1}{2}\delta_{ij}+(H-\lambda)h_{ij}   \nonumber\\
& &\ \ \ \ \ \ \ \ \ \ \ =\frac{1}{2}\delta_{ij}-\lambda h_{ij}- h_{il}h_{lj}  \nonumber\\
& &\ \ \ \ \ \ \ \ \ \ \ \geq\frac{1}{2}\delta_{ij}-\frac{1}{2}(\lambda^2+h^2_{ij})- h_{il}h_{lj}.   \nonumber\\
\end{eqnarray}
Hence, we get the following lower bound for  the $\infty$ - Bakry - Emery Ricci tensor $Ric_f$ of $\lambda$-hypersurface,
\begin{equation}
Ric_f \geq \frac{1-\lambda^2}{2}-\frac{3}{2}|A|^2.
\end{equation}

\section{ Volume comparison theorem of $\lambda$-hypersurfaces}
The classical volume comparison theorem shows that the volume of any ball is
bounded above by the volume of the corresponding ball in the model space, validating the intuitive picture: the bigger the curvature, the smaller the volume. Moreover, this is much less intuitive, if the volume of a big ball has a lower bound, then all the smaller balls also have lower bounds. It enjoys many geometric and topological applications.

In this section, we will investigate the volume comparison theorem of the complete bounded $\lambda$-hypersurface $M^n$ with $|A|\leq\alpha$, that is, the $\infty$ - Bakry - Emery Ricci tensor
\begin{equation}
Ric_f \geq \frac{1-\lambda^2-3\alpha^2}{2},
\end{equation}
where $\alpha$ is an arbitrary nonnegative constant.
Firstly, we fix a point $p\in M^n$, and let $r(x) = d(p, x)$
be the intrinsic distance from $p$ to $x$. This defines a Lipschitz function on the $\lambda$-hypersurface, which is
smooth except the cut locus of $p$. In geodesic polar coordinates, the volume element
$d vol = dr \wedge A_{\vartheta}(r) d\vartheta$, where $d\vartheta$ is the volume form of the standard $S^{n-1}$.
Let $B(p,R)$ be the geodesic ball of $M^n$ with radius $R$ centered at $p$, the
volume of $B(p,R)$ is defined by
\begin{equation}
Vol(B(p,R))=\int_0^R dr\int_{S^{n-1}(p,r)}A_{\vartheta}(r) d\vartheta.
\end{equation}
where $S^{n-1}(p,r)=\{x\in M| d(p,x)=r\}$.
Let $H(r)$ denote the mean curvature of the geodesic sphere at $p$ with outer normal vector $N$, then we have
\begin{equation}
\triangle r=H(r)=\frac{\partial}{\partial r}log A_\vartheta(r).
\end{equation}
Let $\omega_\alpha(t)$ be the solution to
\begin{equation}
\omega_\alpha^{''} + \frac{1-\lambda^2-3\alpha^2}{2(n-1)}\omega_\alpha = 0
\end{equation}
such that $\omega_\alpha(0) = 0$ and $\omega_\alpha^{'}(0) = 1$, i.e. $\omega_\alpha$ are the coefficients of the Jacobi
fields of the simply connected model space $M^n_{\alpha,\lambda}$ with constant curvature $\frac{1-\lambda^2-3\alpha^2}{2(n-1)}$, and
\begin{equation}
\omega_\alpha(t) =
\left\{
\begin{array}{ll}
\frac{\sqrt{2(n-1)}}{\sqrt{1-\lambda^2-3\alpha^2}}sin(\sqrt{\frac{1-\lambda^2-3\alpha^2}{2(n-1)}} \ t),\ \lambda^2+3\alpha^2<1 \\
t, \ \lambda^2+3\alpha^2=1 \\
\frac{\sqrt{2(n-1)}}{\sqrt{|1-\lambda^2-3\alpha^2|}}sinh(\sqrt{\frac{|1-\lambda^2-3\alpha^2|}{2(n-1)}}\ t), \ \lambda^2+3\alpha^2>1\\
\end{array}
\right.
\end{equation}
 Let
$d vol_\alpha = dr \wedge A_{\vartheta_\alpha}(r) d\vartheta_\alpha$ be the volume element of model space $M^n_\alpha$, and denote by $H_{\alpha}$ the mean curvature of the geodesic sphere, then we have
\begin{equation}
H_\alpha(r)=\frac{A^{'}_{\vartheta_\alpha}(r)}{A_{\vartheta_\alpha}(r)}=(n-1)\frac{\omega_\alpha^{'}(r)}{\omega_\alpha(r) }.
\end{equation}
For real numbers $\alpha, \lambda, k$ and $n$, let
\begin{equation}
V_{\alpha,\lambda,k}(r)= vol(S^{n-1}(1))\int_0^r (A_{\vartheta_\alpha}(t))^{(1+\frac{k^2}{2(n-1)})}dt.
\end{equation}
Then we have the following volume comparison theorem for complete bounded $\lambda$-hypersurfaces.
\begin{thm}
Let $X : M \rightarrow \overline{\mathcal{B}}^{n+1}_{k} (0) \subset R^{n+1}$ be  an $n$-dimensional complete $\lambda$-hypersurface with $|A|\leq\alpha$, where $\alpha$ is an arbitrary nonnegative constant, and $\overline{\mathcal{B}}^{n+1}_k(0)$ denotes the Euclidean closed ball with center $0$ and radius $k$. Then for any point $p\in M^n$, and $ 0 < R_1\leq R_2$, we have
\begin{equation}
\frac{Vol(B(p,R_2))}{Vol(B(p,R_1))}\leq e^{\frac{k^2}{4}}\frac{V_{\alpha,\lambda,k}(R_2)}{V_{\alpha,\lambda,k}(R_1)},
\end{equation}
where we assume $R_2 \leq \frac{\sqrt{2(n-1)}\pi}{4\sqrt{1-\lambda^2-3\alpha^2}}$, if $\lambda^2+3\alpha^2<1$. \\
\end{thm}
\begin{proof}
(of Theorem 3.1)
Fix $p \in M$ as a base point, and let $\gamma : [0, r] \rightarrow M$ be a minimizing unit speed
geodesic from $p$. Let $\{E_i(t)\}_{i=1}^{n-1}$ be parallel orthonormal vector fields along
$\gamma(t) $ which are orthogonal to $\dot{\gamma}$. Constructing vector fields $\{X_i(t)= \frac{\omega_\alpha(t)}{|\omega_\alpha(r)|} E_i(t)\}_{i=1}^{n-1}$
along $\gamma $, then by the second variation formula, we have
\begin{eqnarray}
& &\triangle r\leq \int^r_0 \sum_{i=1}^{n-1}(|\nabla_{\dot{\gamma}}X_i|^2-<X_i,R_{X_i,\dot{\gamma}}\dot{\gamma}>)dt            \nonumber \\
& &\ \ \ \ \ =\frac{1}{\omega_\alpha^2(r)}\int^r_0 ((n-1)(\omega^{'}_\alpha(t))^2-\omega_\alpha^2(t) Ric(\dot{\gamma},\dot{\gamma}))dt.            \nonumber \\
\end{eqnarray}
By (2.8), (2.10) and the assumption $|A|\leq\alpha$, we have
\begin{eqnarray}
& &\triangle r\leq \frac{1}{\omega_\alpha^2(r)}\int^r_0 ((n-1)(\omega^{'}_\alpha(t))^2-\omega_\alpha^2(t)Ric_f(\dot{\gamma},\dot{\gamma}) )dt       \nonumber \\
& &\ \ \ \ \ \ \ +\frac{1}{\omega_\alpha^2(r)}\int^r_0\omega_\alpha^2(t)(Hess(f)(\dot{\gamma},\dot{\gamma}))dt           \nonumber \\
& &\ \ \ \ \ \leq\frac{1}{\omega_\alpha^2(r)}\int^r_0 ((n-1)((\omega^{'}_\alpha(t))^2-\frac{1-\lambda^2-3\alpha^2}{2(n-1)}\omega_\alpha^2(t)))dt           \nonumber \\
& &\ \ \ \ \ \ \ +\frac{1}{\omega_\alpha^2(r)}\int^r_0(\omega_\alpha^2(t) \frac{d^2}{d t^2}(f\circ \gamma))dt.  \nonumber \\
\end{eqnarray}
On the other hand, by (3.5) and (3.6), we can get
\begin{equation}
(\omega_\alpha^2(t)H_{\alpha}(t))^{'} = (n-1)((\omega^{'}_\alpha(t))^2-\frac{1-\lambda^2-3\alpha^2}{2(n-1)}\omega_\alpha^2(t)).
\end{equation}
Thus, the inequality (3.10) becomes
\begin{eqnarray}
& &\triangle r\leq \frac{1}{\omega_\alpha^2(r)}\int^r_0 (\omega_\alpha^2(t)H_{\alpha}(t))^{'}dt  +\frac{1}{\omega_\alpha^2(r)}\int^r_0 (\omega_\alpha^2(t)\frac{d^2}{d t^2}(f\circ \gamma))dt        \nonumber \\
& &\ \ \ \ \ =H_{\alpha}(r)+ \frac{1}{\omega_\alpha^2(r)}\int^r_0 (\omega_\alpha^2(t)\frac{d^2}{d t^2}(f\circ \gamma))dt,         \nonumber \\
\end{eqnarray}
where we have used $\omega_\alpha(0) = 0$. By integration by parts on the last term, the expression (3.12) can be written as
\begin{eqnarray}
& &\triangle r\leq H_{\alpha}(r)- \frac{1}{\omega_\alpha^2(r)}\int^r_0 (\omega_\alpha^2(t))^{'}\frac{d}{d t}(f\circ \gamma) dt+ \frac{d}{d t}(f\circ \gamma)|_{t=r}              \nonumber \\
& &\ \ \ \ \ =H_{\alpha}(r)+ \frac{1}{\omega_\alpha^2(r)}(\int^r_0 (\omega_\alpha^2(t))^{''}(f\circ \gamma) dt- (\omega_\alpha^2(r))^{'}(f\circ \gamma)(r))\nonumber \\
& &\ \ \ \ \ \ \ \ + \frac{d}{d t}(f\circ \gamma)|_{t=r}.          \nonumber \\
\end{eqnarray}
Since the $\lambda$-hypersurface is contained in the ball $\mathcal{B}^{n+1}_k(0)$, then $f = \frac{|X|^2}{4} \leq \frac{k^2}{4}$. On the other hand, we assume $r\in(0,  \frac{\sqrt{2(n-1)}\pi}{4\sqrt{1-\lambda^2-3\alpha^2}}]$, if $\alpha<\sqrt{\frac{1-\lambda^2}{3}}$. This implies that $(\omega_\alpha^2)^{'}\geq 0$ and $(\omega_\alpha^2)^{''}\geq 0$, for all $\alpha\geq 0$. Then
\begin{eqnarray}
& &\triangle r\leq H_{\alpha}(r)+ \frac{1}{\omega_\alpha^2(r)}(\frac{k^2}{4}\int^r_0 (\omega_\alpha^2(t))^{''}dt- (\omega_\alpha^2(r))^{'}(f\circ \gamma)(r))\nonumber \\
& &\ \ \ \ \ \ \ \ + \frac{d}{d t}(f\circ \gamma)|_{t=r}          \nonumber \\
& &\ \ \ \ \ = H_{\alpha}(r)+ \frac{(\omega_\alpha^2(r))^{'}}{\omega_\alpha^2(r)}(\frac{k^2}{4} - (f\circ \gamma)(r))+ \frac{d}{d t}(f\circ \gamma)|_{t=r}          \nonumber \\
& &\ \ \ \ \ \leq H_{\alpha}(r)+ \frac{k^2}{4}\frac{(\omega_\alpha^2(r))^{'}}{\omega_\alpha^2(r)}+\frac{d}{d t}(f\circ \gamma)|_{t=r}       \nonumber \\
& &\ \ \ \ \ =  H_{\alpha}(r)+ \frac{k^2}{2(n-1)}H_{\alpha}(r)+\frac{d}{d t}(f\circ \gamma)|_{t=r}            \nonumber \\
& &\ \ \ \ \ =  (1+ \frac{k^2}{2(n-1)})H_{\alpha}(r)+\frac{d}{d t}(f\circ \gamma)|_{t=r},             \nonumber \\
\end{eqnarray}
here in the fifth line we have used (3.6). By (3.3) and (3.6), the expression (3.14) can be written as
\begin{equation}
\frac{\partial}{\partial r}log A_\vartheta(r) \leq  (1+ \frac{k^2}{2(n-1)})\frac{d}{d r}log A_{\vartheta_\alpha}(r)+\frac{d}{d t}(f\circ \gamma)|_{t=r}.
\end{equation}
Integrating it from $s_1$ to $s_2$, together with $f = \frac{|X|^2}{4} \leq \frac{k^2}{4}$, we have
\begin{equation}
log \frac{A_\vartheta(s_2)}{A_\vartheta(s_1)} \leq  (1+ \frac{k^2}{2(n-1)})log \frac{A_{\vartheta_\alpha}(s_2)}{A_{\vartheta_\alpha}(s_1)}+\frac{k^2}{4}.
\end{equation}
Then
\begin{equation}
\frac{A_\vartheta(s_2)}{A_\vartheta(s_1)} \leq  e^{\frac{k^2}{4}}(\frac{A_{\vartheta_\alpha}(s_2)}{A_{\vartheta_\alpha}(s_1)})^{(1+ \frac{k^2}{2(n-1)})},
\end{equation}
that is,
\begin{equation}
A_\vartheta(s_2) \cdot (A_{\vartheta_\alpha}(s_1))^{(1+ \frac{k^2}{2(n-1)})}  \leq e^{\frac{k^2}{4}} \cdot A_\vartheta(s_1) \cdot (A_{\vartheta_\alpha}(s_2))^{(1+ \frac{k^2}{2(n-1)})}.
\end{equation}
Integrating it from $0$ to $R_1$ with respect to $s_1$, and from $0$ to $R_2$ with respect to $s_2$, we have
\begin{eqnarray}
& &\ \ \ \ \int_0^{R_2}A_\vartheta(s_2)d s_2 \cdot \int_0^{R_1}(A_{\vartheta_\alpha}(s_1))^{(1+ \frac{k^2}{2(n-1)})}d s_1   \nonumber \\
& &\leq e^{\frac{k^2}{4}} \cdot \int_0^{R_1}A_\vartheta(s_1)d s_1  \cdot \int_0^{R_2}(A_{\vartheta_\alpha}(s_2))^{(1+ \frac{k^2}{2(n-1)})}d s_2 .   \nonumber \\
\end{eqnarray}
Integration along the sphere direction gives
\begin{equation}
Vol(B(p,R_2))\cdot V_{\alpha,\lambda,k}(R_1)\leq e^{\frac{k^2}{4}}\cdot Vol(B(p,R_1)) \cdot V_{\alpha,\lambda,k}(R_2).
\end{equation}
Then the result follows.
\end{proof}

Note that the $\lambda$-hypersurface is a self-shrinker of the mean curvature flow when $\lambda= 0$, then we have the following corollary.
\begin{cor}
Let $X : M \rightarrow \overline{\mathcal{B}}^{n+1}_{k} (0) \subset R^{n+1}$ be  an $n$-dimensional complete self-shrinker with $|A|\leq\alpha$. Then for any $p\in M^n$, $0 < R_1\leq R_2$, we have
\begin{equation}
\frac{Vol(B(p,R_2))}{Vol(B(p,R_1))}\leq e^{\frac{k^2}{4}}\frac{V_{\alpha,0,k}(R_2)}{V_{\alpha,0,k}(R_1)},
\end{equation}
where $R_2 \leq \frac{\sqrt{2(n-1)}\pi}{4\sqrt{1-\lambda^2-3\alpha^2}}$ if $\lambda^2+3\alpha^2<1$.
\end{cor}
Since $(A_{\vartheta_\alpha}(t))^{(1+\frac{k^2}{2(n-1)})}=(\omega_\alpha(t))^{(n-1)(1+\frac{k^2}{2(n-1)})}=(\omega_\alpha(t))^{(n+\frac{k^2}{2}-1)}$,
then $dr \wedge A_{\vartheta_\alpha}(r)^{(1+\frac{k^2}{2(n-1)})} d\vartheta_\alpha$ can be considered as the volume element of simply connected model space $M^{n+\frac{k^2}{2}}_{\alpha,\lambda}$ of dimension $n+\frac{k^2}{2}$ with constant curvature $\frac{1-\lambda^2-3\alpha^2}{2(n-1)}$. Now by $(3.18)$, we can obtain the following volume comparison for balls.
\begin{thm}
Let $X : M \rightarrow \overline{\mathcal{B}}^{n+1}_{k} (0) \subset R^{n+1}$ be  an $n$-dimensional complete $\lambda$-hypersurface with $|A|\leq\alpha$. Then for any $p\in M^n$, $ 0 < R_1\leq R_2$, we have
\begin{equation}
\frac{Vol(B(p,R_2))}{Vol(B(p,R_1))}\leq e^{\frac{k^2}{4}}\frac{Vol_{\alpha,\lambda}^{n+\frac{k^2}{2}}(R_2)}{Vol_{\alpha,\lambda}^{n+\frac{k^2}{2}}(R_1)},
\end{equation}
where $Vol_{\alpha,\lambda}^{n+\frac{k^2}{2}}(r)$ is the volume of the ball with radius $r$ in model space $M^{n+\frac{k^2}{2}}_{\alpha,\lambda}$, and $R_2 \leq \frac{\sqrt{2(n-1)}\pi}{4\sqrt{1-\lambda^2-3\alpha^2}}$ if $\lambda^2+3\alpha^2<1$.
\end{thm}
Actually,  it is not easy to figure out the relevant conclusions by Theorem 3.1 and Theorem 3.3. In particular, for the complete $\lambda$-hypersurfaces with $|A|\leq \sqrt{\frac{1-\lambda^2}{3}}$(i.e. $Ric_f \geq 0$ and $|\lambda|\leq 1$ ), we obtain the following interesting result.
\begin{thm}
Let $X : M \rightarrow \overline{\mathcal{B}}^{n+1}_{k} (0) \subset R^{n+1}$ be  an $n$-dimensional complete $\lambda$-hypersurface with $|A|\leq\sqrt{\frac{1-\lambda^2}{3}}$. Then for any $p\in M^n$, $ 0 < R_1\leq R_2$, we have
\begin{equation}
\frac{Vol(B(p,R_2))}{Vol(B(p,R_1))}\leq e^{\frac{3k^2}{4}}\frac{V(R_2)}{V(R_1)},
\end{equation}
where $V(r)$ is the volume of the ball with radius $r$ in Euclidean space $R^{n}$. Moreover, we can get
\begin{equation}
Vol(B(p,R_2))\leq e^{\frac{3k^2}{4}} V(R_2).
\end{equation}
\end{thm}
\begin{proof}
(of Theorem 3.4) Since $|A| \leq\sqrt{\frac{1-\lambda^2}{3}}$ implies $Ric_f \geq 0$, then the expression (3.13) can be written as
\begin{eqnarray}
& &\triangle r\leq \frac{n-1}{r}+\frac{1}{r^2}(\int^r_0 (t^2)^{''}(f\circ \gamma)dt-((t^2)^{'} f\circ \gamma)|_{t=r})+ \frac{d}{d t}(f\circ \gamma)dt|_{t=r}          \nonumber \\
& &\ \ \ \ \ =\frac{n-1}{r}+\frac{2}{r^2}\int^r_0(f\circ \gamma)dt-\frac{2}{r} f\circ \gamma(r)+ \frac{d}{d t}(f\circ \gamma)dt|_{t=r}.         \nonumber \\
\end{eqnarray}
That is,
\begin{equation}
\frac{\partial}{\partial r}log A_\vartheta(r) \leq  \frac{d}{d r}log   r^{n-1}+\frac{2}{r^2}\int^r_0(f\circ \gamma)dt-\frac{2}{r} f\circ \gamma(r)+ \frac{d}{d t}(f\circ \gamma))dt|_{t=r}.
\end{equation}
Integrating from $s_1$ to $s_2$, together with $f = \frac{|X|^2}{4} \leq \frac{k^2}{4}$, we have
\begin{eqnarray}
& &log \frac{A_\vartheta(s_2)}{A_\vartheta(s_1)} \leq  log \frac{s_2^{n-1}}{s_1^{n-1}}+\int^{s_2}_{s_1}(\frac{2}{r^2}\int^r_0(f\circ \gamma)dt)dr-\int^{s_2}_{s_1}(\frac{2}{r} f\circ \gamma(r))dr \nonumber \\
& &\ \ \ \ \ \ \ \ \  \ \ \ \ \ \ \ \ +f\circ \gamma|_{s_1}^{s_2} \nonumber \\
& &\ \ \ \ \ \ \ \ \  \ \ \ \ = log \frac{s_2^{n-1}}{s_1^{n-1}}-\frac{2}{r}\int^r_0(f\circ \gamma)dt|_{s_1}^{s_2}+f\circ \gamma|_{s_1}^{s_2} \nonumber \\
& &\ \ \ \ \ \ \ \ \  \ \ \ \ \leq log \frac{s_2^{n-1}}{s_1^{n-1}}+\frac{k^2}{2}+\frac{k^2}{4} \nonumber \\
& &\ \ \ \ \ \ \ \ \  \ \ \ \ = log \frac{s_2^{n-1}}{s_1^{n-1}}+\frac{3k^2}{4}. \nonumber \\
\end{eqnarray}
Now the result is obvious.
\end{proof}

\section{Some applications of volume comparison theorem}
The classical volume comparison theorem is a powerful tool in studying the manifolds with lower Ricci
curvature bound(See [36]).  In this section, we will give some applications of
the volume comparison of $\lambda$-hypersurfaces with lower $\infty$ - Bakry - Emery Ricci
curvature bound.

Firstly, we obtain the lower bound and upper bound on
volume growth for $\lambda$-hypersurface with $Ric_f \geq 0$ (i.e. $|A|\leq \sqrt{\frac{1-\lambda^2}{3}}$).
\begin{thm}
Let $X : M \rightarrow \overline{\mathcal{B}}^{n+1}_{k} (0) \subset R^{n+1}$ be  an $n$-dimensional complete $\lambda$-hypersurface with $|A|\leq \sqrt{\frac{1-\lambda^2}{3}}$. Then, for any $p \in M$,
\begin{equation}
C_1 R^{n} \geq Vol(B(p,R))\geq C_2 R.
\end{equation}
\end{thm}
\begin{proof}
(of Theorem 4.1) Since $|A|\leq \sqrt{\frac{1-\lambda^2}{3}}$, for any $p \in M$,
by Theorem 3.4, we have
\begin{equation}
Vol(B(p,R))\leq C_1 R^{n}.      \\
\end{equation}
On the other hand, for $T_1\leq T_2, R_1\leq R_2, T_1\leq  R_1, T_2\leq R_2,$ and $q \in M$,  we let  $A_q(R_1,R_2)$ be the set of $x \in M$ such that $R_1\leq r(x) \leq R_2$ and $V(T_1,T_2)= vol(S^{n-1}(1))\int_{T_1}^{T_2} t^{n-1}dt$, where $r(x)=d(q,x)$. Then
 by (3.27), we can get
\begin{equation}
\frac{Vol(A_q(R_1,R_2))}{Vol(A_q(T_1,T_2))}\leq e^{\frac{3k^2}{4}}\frac{V(R_1,R_2)}{V(T_1,T_2)}.
\end{equation}
Let $\Gamma$ be a geodesic based at $p$ in $M$, by $|A|\leq \sqrt{\frac{1-\lambda^2}{3}}$ and the annulus relative volume comparison
(4.3) to annuli centered at $\Gamma(t)$, we have
\begin{eqnarray}
& &\frac{Vol(A_{\Gamma(t)}(t-1,t+1))}{Vol(B(\Gamma(t),t-1))}\leq e^{\frac{3k^2}{4}}\frac{(t+1)^{n}-(t-1)^{n}}{(t-1)^{n}}      \nonumber \\
& &\ \ \ \ \ \ \ \ \ \ \ \ \ \ \ \ \  \ \ \ \ \ \ \ \ \ \ \ \ \ \   =\frac{\tilde{c}(n,k)}{t}.     \nonumber \\
\end{eqnarray}
By $B(\Gamma(0),1)\subset A_{\Gamma(t)}(t-1,t+1)$, we get
\begin{eqnarray}
& &Vol(B(\Gamma(t),t-1))\geq \frac{t}{\tilde{c}(n,k)}Vol(A_{\Gamma(t)}(t-1,t+1))      \nonumber \\
& &\ \ \ \ \ \ \ \ \ \ \ \ \ \ \ \ \  \ \ \ \ \ \ \ \  \geq \frac{t}{\tilde{c}(n,k)}Vol(B(p,1))    \nonumber \\
& &\ \ \ \ \ \ \ \ \ \ \ \ \ \ \  \ \ \  \ \ \ \ \ \ \  =\tilde{\tilde{c}} t.     \nonumber \\
\end{eqnarray}
Therefore,
\begin{equation}
Vol(B(p,R))\geq Vol(B(\Gamma(\frac{R}{2}),\frac{R}{2}-1))\geq C_2 R.
\end{equation}
Combining (4.2) and (4.6), the result follows easily.
\end{proof}

In [21], Milnor observed that polynomial volume growth on the universal cover of a manifold restricts the structure of its fundamental group. Thus Theorem 3.4 also implies the following extension of Milnor's Theorem.
\begin{thm}
Let $X : M \rightarrow \overline{\mathcal{B}}^{n+1}_{k} (0) \subset R^{n+1}$ be  an $n$-dimensional complete $\lambda$-hypersurface with $|A|\leq \sqrt{\frac{1-\lambda^2}{3}}$. Then any finite generated subgroup of  the fundamental group
of $M$ has polynomial
growth of order at most $n$.
\end{thm}
Note that the relation between the fundamental group and the first Betti number given by the Hurewicz Theorem
[34], Ricci curvature can also give control on the first Betti number. By the same assertions as in M. Gromov [13], we can prove the following theorem by Theorem 3.1.
\begin{thm}
Let $X : M \rightarrow \overline{\mathcal{B}}^{n+1}_{k} (0) \subset R^{n+1}$ be  an $n$-dimensional complete $\lambda$-hypersurface with $|A|\leq \alpha$, $diam_M\leq d$. Then  $b_1(M) \leq c(n,k,\alpha,d^2)$.
\end{thm}
Theorem 3.1 also implies the following extensions
theorems  of Anderson [1]. We will leave it for readers.
\begin{thm}
For the class of $n$-dimensional $\lambda$-hypersurfaces $M$ with $|A|\leq \alpha, vol\geq v>0, diam \leq d,$ which contained in the compact ball $\overline{\mathcal{B}}^{n+1}_k(0) \subset R^{n+1}$, there are only finitely many isomorphism classes of the fundamental group of $M$.
\end{thm}
The volume comparison has many other geometric applications, such as, in the Gromov-Hausdorff convergence theory, in the rigidity and pinching theory. We will leave these statements to the interested readers.

\section{Estimate of the exterior radius}
The $n$-dimensional sphere $S^{n}(k)$ with radius $k=\sqrt{\lambda^2+2n}+\lambda$ is a compact $\lambda$-hypersurface and contained in the compact closed ball $\overline{\mathcal{B}}^{n+1}_k(0)\subset R^{n+1}$.
Our first remark is that if a complete $\lambda$-hypersurface  with controlled intrinsic
volume growth is contained in some Euclidean closed ball $\overline{\mathcal{B}}^{n+1}_k(0)$ with center $0$ and radius $k$, then there is an obvious
relation among $\lambda$, the radius $k$ and the dimension $n$. To prove this, we need the following elementary lemma.
\begin{lem}
(see[27])Let $(M, dvol_f)$ be a geodesically complete weighted manifold satisfying the volume growth condition
\begin{equation}
\frac{R}{log(Vol_f B(p,R))}
\notin L^1(+\infty),
\end{equation}
Then the weak maximum principle at infinity for the $f$-Laplacian holds on $M$.
\end{lem}
\begin{thm}
Let $X : M \rightarrow \overline{\mathcal{B}}^{n+1}_{k} (0) \subset R^{n+1}$ be  an $n$-dimensional complete $\lambda$-hypersurface whose intrinsic volume growth satisfies
\begin{equation}
\frac{R}{log(Vol B(p,R))}
\notin L^1(+\infty),
\end{equation}
where $B(p,R)$ is the geodesic ball of $\lambda$-hypersurface $M^n$ with radius $R$ centered at $p$, and $\mathcal{B}^{n+1}_k(0)$ denotes the Euclidean ball with center $0$ and radius $k$. Then
\begin{equation}
k\geq \sqrt{\lambda^2+2n}-|\lambda|.
\end{equation}
\end{thm}
\begin{proof}
(of Theorem 5.2)
Since $\triangle X = - H N$ on the hypersurface and $H-\frac{<N,X>}{2}=\lambda$ on the $\lambda$-hypersurface, then
\begin{eqnarray}
& &\triangle |X|^2=2|\nabla X|^2+2<\triangle X, X>        \nonumber \\
& &\ \ \ \ \ \ \ \ \ =2|\nabla X|^2-2H<N, X>     \nonumber \\
& &\ \ \ \ \ \ \ \ \ =2|\nabla X|^2-2(\frac{<N,X>}{2}+\lambda)<N, X>     \nonumber \\
& &\ \ \ \ \ \ \ \ \ =2n-2\lambda<N, X>- <N,X>^{2}. \nonumber \\
\end{eqnarray}
Also note that $\nabla|X|^2 = 2 X^T$, where $X^T$ is the tangential projection of $X$, we can get
\begin{eqnarray}
& &\triangle_f |X|^2=\triangle |X|^2- <\nabla |X|^2,\nabla f>       \nonumber \\
& &\ \ \ \ \ \ \ \ \ \ \  =\triangle |X|^2- <\nabla |X|^2,\nabla \frac{|X|^2}{4}>     \nonumber \\
& &\ \ \ \ \ \ \ \ \ \ \  =\triangle |X|^2- \frac{1}{4}<2 X^T,2 X^T>     \nonumber \\
& &\ \ \ \ \ \ \ \ \ \ \  =2n-2\lambda<N, X>- <N,X>^{2}- |X^T|^2     \nonumber \\
& &\ \ \ \ \ \ \ \ \ \ \  =2n-2\lambda<N, X>- |X|^2 \nonumber \\
& &\ \ \ \ \ \ \ \ \ \ \  \geq 2n-2|\lambda||X|- |X|^2. \nonumber \\
\end{eqnarray}
On the other hand, since $c^{-1} d vol_f \leq d vol \leq c d vol_f$
for a large enough constant $c > 1$, then
\begin{equation}
R \rightarrow \frac{R}{log(Vol_f B(p,R))}
\notin L^1(+\infty),
\end{equation}
which implies that on the $\lambda$-hypersurface the weak maximum
principle holds at infinity for the drifted Laplacian $\triangle_f$ (Lemma 5.1). Therefore
\begin{equation}
0 \geq 2n- 2|\lambda|sup_M|X|- sup_M|X|^2\geq 2n- 2|\lambda|k- k^2,
\end{equation}
and the claimed lower estimate on $k$ follows. This completes the proof of Theorem 5.2.
\end{proof}

By(4.1), we can specialize Theorem 5.2 to the following
\begin{cor}
Let $X : M \rightarrow \overline{\mathcal{B}}^{n+1}_{k} (0) \subset R^{n+1}$ be  an $n$-dimensional complete $\lambda$-hypersurface with $|A|\leq \sqrt{\frac{1-\lambda^2}{3}}$, then
\begin{equation}
k\geq \sqrt{\lambda^2+2n}-|\lambda|.
\end{equation}
\end{cor}

\section{Estimate of the intrinsic diameter}
The purpose of this section  is to generalize the well-known Myers' theorem [23]
on a complete and connected $\lambda$-hypersurface with $Ric_f \geq \frac{1-\lambda^2-3\alpha^2}{2} >0$ (i.e. $|A|\leq\alpha<\sqrt{\frac{1-\lambda^2}{3}}$).  In particular, we obtain the following
\begin{thm}
Let $X : M \rightarrow \overline{\mathcal{B}}^{n+1}_{k} (0) \subset R^{n+1}$ be  an $n$-dimensional complete $\lambda$-hypersurface with $|A|\leq \alpha < \sqrt{\frac{1-\lambda^2}{3}}$. Then $M$ is compact and the intrinsic diameter satisfies
\begin{equation}
diam(M) \leq \sqrt{\frac{2(n-1)(n+k^2-1)}{1-\lambda^2-3\alpha^2}}\pi.
\end{equation}
\end{thm}
\begin{proof}
(of Theorem 6.1) The proof goes by contradiction. If  $\gamma: [0,l] \rightarrow M$ is a unit speed geodesic of length $l >\sqrt{\frac{2(n-1)(n+k^2-1)}{1-\lambda^2-3\alpha^2}}\pi$.
 Let $\{E_i(t)\}_{i=1}^{n-1}$ be parallel orthonormal vector fields along
$\gamma $ which are orthogonal to $\dot{\gamma}$. Using vector fields $\{X_i(t)= \sin(\frac{\pi}{l}t)  E_i(t)\}_{i=1}^{n-1}$
along $\gamma $, then we get the index form
\begin{equation}
I(X_i(t),X_i(t))=\int^l_0 (|\nabla_{\dot{\gamma}}X_i|^2-<X_i,R_{X_i,\dot{\gamma}}\dot{\gamma}>)dt.
\end{equation}
Then by $|A|\leq \alpha < \sqrt{\frac{1-\lambda^2}{3}}$, $f = \frac{|X|^2}{4} \leq \frac{k^2}{4}$, $l >\sqrt{\frac{2(n-1)(n+k^2-1)}{1-\lambda^2-3\alpha^2}}\pi$,  (2.8) and  (2.10), we have
\begin{eqnarray}
& & \ \ \ \  \sum_{i=1}^{n-1}I(X_i(t),X_i(t))  \nonumber \\
& &=\int^l_0 \sum_{i=1}^{n-1}(|\nabla_{\dot{\gamma}}X_i|^2-<X_i,R_{X_i,\dot{\gamma}}\dot{\gamma}>)dt            \nonumber \\
& &=\int^l_0((n-1)(\frac{\pi}{l}cos(\frac{\pi}{l}t))^2- (\sin(\frac{\pi}{l}t))^2Ric(\dot{\gamma},\dot{\gamma}))dt\nonumber \\
& & \leq\int^l_0((n-1)(\frac{\pi}{l}cos(\frac{\pi}{l}t))^2- (\sin(\frac{\pi}{l}t))^2(\frac{1-\lambda^2-3\alpha^2}{2(n-1)}))dt\nonumber \\
& & \ \ \ \  +\int^l_0(\sin(\frac{\pi}{l}t))^2(Hess(f)(\dot{\gamma},\dot{\gamma}))dt \nonumber \\
& &=\frac{(n-1)\pi^2}{2l}-\frac{(1-\lambda^2-3\alpha^2)l}{4(n-1)}+\int^l_0(\sin(\frac{\pi}{l}t))^2\frac{\partial^2}{\partial t^2}(f\circ\gamma) dt\nonumber \\
& &=\frac{(n-1)\pi^2}{2l}-\frac{(1-\lambda^2-3\alpha^2)l}{4(n-1)}-\frac{\pi}{l}
\int^l_0\sin(\frac{2\pi}{l}t)\frac{\partial}{\partial t}(f\circ\gamma) dt\nonumber \\
& &=\frac{(n-1)\pi^2}{2l}-\frac{(1-\lambda^2-3\alpha^2)l}{4(n-1)}+ \frac{2\pi^2}{l^2}
\int^l_0\cos(\frac{2\pi}{l}t)(f\circ\gamma) dt\nonumber \\
& &\leq \frac{(n-1)\pi^2}{2l}-\frac{(1-\lambda^2-3\alpha^2)l}{4(n-1)}+ \frac{2\pi^2}{l^2}
\int^l_0\ |(f\circ\gamma)| dt\nonumber \\
& &\leq \frac{(n-1)\pi^2}{2l}-\frac{(1-\lambda^2-3\alpha^2)l}{4(n-1)}+ \frac{\pi^2k^2}{2l}
\nonumber \\
& &=\frac{(n+k^2-1)\pi^2}{2l}-\frac{(1-\lambda^2-3\alpha^2)l}{4(n-1)} \nonumber \\
& &< 0.\nonumber \\
\end{eqnarray}
This implies $I(X_i(t),X_i(t))< 0$, for some $1\leq i\leq n-1$. Namely, the index form  is not positive semi-definite. It is a contradiction. So we finish the proof.
\end{proof}

\section{Topology at infinity of $\lambda$-hypersurfaces }
In this section, by the following  Cheeger-Gromoll-Lichnerowicz splitting theorem, we obtain a bit of information on the topology at infinity of a bounded $\lambda$-hypersurfaces with $Ric_f\geq 0$.
\begin{lem}
(see[17])Let $(M^n, dvol_f)$ be a geodesically complete weighted manifold with $Ric_f\geq 0$ for some
bounded function $f$ and $M^n$ contains a line, then $M^n = N^{n-1} \times R$ and $f$ is constant along the line.
\end{lem}
\begin{thm}
Let $X : M \rightarrow \overline{\mathcal{B}}^{n+1}_{k} (0) \subset R^{n+1}$ be  an $n$-dimensional complete non-compact $\lambda$-hypersurface with $|A|\leq \sqrt{\frac{1-\lambda^2}{3}}$. Then $M$ does not contain a line. In particular, $M$ is connected at infinity, i.e., $M$ has only one end.
\end{thm}
\begin{proof}
(of Theorem 7.2) The proof goes by contradiction.  By $|A|\leq \sqrt{\frac{1-\lambda^2}{3}}$, we have $Ric_f\geq 0$. If $M$ contains a line, by Lemma 7.1, $M$ can be split isometrically as the Riemannian product $(N^{n-1} \times R, g_N + dt\otimes dt)$, and $f$ is constant along the line. Hence,
\begin{equation}
Hess(f)(\frac{\partial}{\partial t}, \frac{\partial}{\partial t}) = 0.
\end{equation}
On the other hand, $H-\frac{<N,X>}{2}=\lambda$ implies that
\begin{equation}
H_i=-\frac{1}{2}\sum_jh_{ij}<X,e_j>,
\end{equation}
and
\begin{equation}
H_{ik}=-\frac{1}{2}(\sum_j(h_{ijk}<X,e_j>+h_{ik}+\sum_jh_{ij}h_{kj}(2H-2\lambda)).
\end{equation}
Therefore,
\begin{eqnarray}
& &\triangle h_{ij}=-|A|^2h_{ij}-H h_{ik}h_{kj}- H_{ij}         \nonumber \\
& &\ \ \ \ \ \ \  =-|A|^2h_{ij}-H h_{ik}h_{kj}+\frac{1}{2}(\sum_k(h_{ikj}<X,e_k>+h_{ij}    \nonumber \\
& &\ \ \ \ \ \ \ \ \ \ +\sum_k h_{ik}h_{jk}(2H-2\lambda))     \nonumber \\
& &\ \ \ \ \ \ \  =-|A|^2h_{ij}+\sum_kh_{ikj}<\frac{X}{2},e_k>+\frac{1}{2}h_{ij}-\lambda\sum_k h_{ik}h_{jk}.    \nonumber \\
\end{eqnarray}
Then it follows that
\begin{equation}
\triangle_f h_{ij}=(\frac{1}{2}-|A|^2)h_{ij}-\lambda\sum_k h_{ik}h_{jk}.
\end{equation}
By $\frac{1}{2}\triangle_f\eta^2=|\nabla \eta|^2+\eta\triangle_f\eta$, we have
\begin{equation}
\frac{1}{2}\triangle_f |A|^2=\sum_{ijk}h_{ijk}^2+(\frac{1}{2}-|A|^2)|A|^2-\lambda\sum_{ijk} h_{ij}h_{jk}h_{ki}.
\end{equation}
Then, by $|A|\leq \sqrt{\frac{1-\lambda^2}{3}}$, we have
\begin{eqnarray}
& &\frac{1}{2}\triangle_f |A|^2\geq\sum_{ijk}h_{ijk}^2+(\frac{1}{2}-|A|^2)|A|^2-|\lambda||A|^3      \nonumber \\
& &\ \ \ \ \ \ \ \ \ \ \ \  =\sum_{ijk}h_{ijk}^2+(\frac{1}{2}-|\lambda||A|-|A|^2)|A|^2     \nonumber \\
& &\ \ \ \ \ \ \ \ \ \ \ \  \geq\sum_{ijk}h_{ijk}^2+(\frac{1-\lambda^2}{2}-\frac{3}{2}|A|^2)|A|^2     \nonumber \\
& &\ \ \ \ \ \ \ \ \ \ \ \  \geq0. \nonumber \\
\end{eqnarray}
Hence, the strong maximum principle for the drifted Laplacian yields that either  $|A| < \sqrt{\frac{1-\lambda^2}{3}} $ or $|A| \equiv \sqrt{\frac{1-\lambda^2}{3}}$.\\
Case 1: If $|A| < \sqrt{\frac{1-\lambda^2}{3}},$ then $|A|=0$ and
\begin{equation}
Hess(f)(\frac{\partial}{\partial t}, \frac{\partial}{\partial t})=Ric_f(\frac{\partial}{\partial t}, \frac{\partial}{\partial t})\geq \frac{(1-\lambda^2-3|A|^2)}{2} >0 ,
\end{equation}
which contradicts (7.1).\\
Case 2: If $|A| \equiv \sqrt{\frac{1-\lambda^2}{3}},$ then $h_{ijk}=0.$ By the classical Lawson's classification theorem, $M$ is  a cylindrical
product $S^k\times R^{n-k}$. Since the $\lambda$-hypersurface is bounded,
we conclude that $M=S^n$, which contradicts the assumption that $M$ is not compact.
Then it completes the proof.
\end{proof}

\section{Compactness of $\lambda$-hypersurfaces}
In this section, we will follow the notations and conclusions of [4]. Then the technique of the Feller property combining with the stochastic completeness, will enable us to get natural and general restrictions that force the $\lambda$-hypersurface to be compact.
\begin{thm}
Let $X : M \rightarrow \overline{\mathcal{B}}^{n+1}_{k} (0) \subset R^{n+1}$ be  an $n$-dimensional complete non-compact $\lambda$-hypersurface. Then we have
\begin{equation}
\lim_{r\rightarrow \infty}\sup_{M\backslash B(r)}|A|\geq \sqrt{\frac{1-\lambda^2}{3}},
\end{equation}
where $B(r)$ is  the geodesic ball of $\lambda$-hypersurface $M^n$ with radius $r$.
\end{thm}
\begin{proof}
(of Theorem 8.1) By contradiction. Suppose that $\displaystyle\lim_{r\rightarrow \infty}\sup_{M\backslash B(r)}|A|< \sqrt{\frac{1-\lambda^2}{3}}$, we have $|A| \in L^\infty(M)$.  Since $|\nabla f|=\frac{1}{2}|X^T|\leq\frac{1}{2}|X|<\frac{1}{2}k$,
then by (2.10), Theorem 7  and  Theorem 8 in $[4]$, $M$ is stochastically complete and Feller with respect to $\triangle_f$.
By the Simons type equation (7.7), we have, for some $\alpha$,
\begin{equation}
\frac{1}{2}\triangle_f |A|^2\geq(\frac{1}{2}-|\lambda|\alpha-\alpha^2)|A|^2,
\end{equation}
outside a smooth domain $\Omega\subset\subset M$. Then Theorem 2 in [4] gives
\begin{equation}
|A|(x) \rightarrow 0, \mbox{\ as\ }  x \rightarrow \infty.
\end{equation}
By (8.3) and $Hess(f)_{ij}=\frac{1}{2}(\delta_{ij}+<X,N>h_{ij})$, we have
\begin{equation}
\frac{d^2}{dt^2}(f\circ \gamma)(t)=Hess(f)(\dot{\gamma}, \dot{\gamma})\geq \frac{1}{4},
\end{equation}
for some ray $\gamma: [0, +\infty) \rightarrow M$, and $t >>1$.
It follows by integration that $|X|^2\rightarrow +\infty$ along $\gamma$, therefore, $M$ is
unbounded. This is a contradiction. So we finish the proof.
\end{proof}
By Theorem 8.1, the following corollary is obvious.
\begin{cor}
Let $X : M \rightarrow \overline{\mathcal{B}}^{n+1}_{k} (0) \subset R^{n+1}$ be  an $n$-dimensional complete  $\lambda$-hypersurface with
\begin{equation}
\lim_{r\rightarrow \infty}\sup_{M\backslash B(r)}|A|< \sqrt{\frac{1-\lambda^2}{3}},
\end{equation}
where $B(r)$ is  the geodesic ball of $\lambda$-hypersurface $M^n$ with radius $r$. Then $M$ is compact.
\end{cor}

% ------------------------------------------------------------------------
\end{document}